\documentclass[10pt]{ifacconf}

\usepackage{natbib}
\usepackage{parskip}
\usepackage{color}
\usepackage{amsmath}
\usepackage{dsfont}
\usepackage{latexsym}
\usepackage{amssymb}
\usepackage{graphicx}
\usepackage{float}
\usepackage{tikz}
\usetikzlibrary{shapes,arrows,fit}
\usepackage{supertabular}
\usepackage{epsfig}
\usepackage[frame]{xy}
\usepackage{fancybox}
\usepackage{tabularx}
\usepackage{multirow}
\usepackage{calc}
\usepackage{bbm}
\usepackage{wasysym}
\usepackage{ifsym}
\def\DZ{{Z\kern-.40em Z}}
\def\DN{{I\kern-.13em  N}}
\def\DR{{I\kern-.13em R}}
\def\DK{{I\kern-.13em K}}
\def\DC{{\kern.24em \vrule width.05em height1.5ex depth-.05ex \kern-.30em C}}
\def\DQ{{\kern.24em \vrule width.05em height1.5ex depth-.05ex \kern-.30em Q}}

\def\qed{\framebox(3,3){}}
\def\be{\begin{equation}}
\def\ee{\end{equation}}
\def\bear{\begin{eqnarray}}
\def\eear{\end{eqnarray}}

\def\eqref#1{(\ref{#1})}

\newtheorem{lemma}{Lemma}
\newtheorem{problem}{Problem}

\newenvironment{proof}
{{\it Proof.} } {\hfill \large \qed\bigskip\par}

\begin{document}
\begin{frontmatter}
\title{A note on the state-space realizations equivalence}

\author[First]{P. Lopes dos Santos}
\author[Second]{J. A. Ramos}
\author[Third]{ T-P Azevedo Perdico\'{u}lis}
\author[Fourth]{J. L. Martins de Carvalho}
\address[First]{FEUP, Portugal (e-mail:  pjsantos@fe.up.pt)}
\address[Second]{Farquhar College of Arts and Sciences,USA (e-mail:  jr1284@nova.edu)}
\address[Third]{ISR---Coimbra \&  UTAD, Portugal (e-mail:  tazevedo@utad.pt)}
\address[Fourth]{FEUP, Portugal (e-mail:  jmartins@fe.up.pt)}

\maketitle

\begin{abstract}
In this note we resolve the problem of  getting a state-space realization \emph{compatible} with the deterministic state-space realization
and the filtering problem.
\end{abstract}
\end{frontmatter}

\section{}\label{}
\begin{problem}
Given $A_0 \in\mathbb{R}^{n_x\times n_x},\;C_0 \in \mathbb{R}^{1 \times n_x},\; A\in\mathbb{R}^{1\times n_x}$ and $C \in\mathbb{R}^{n_x\times n_x},$
 such that
\begin{eqnarray}
& & \left(A_0,C_0\right) \text{ and } \left(A,C\right)\qquad \text{are observable}, \label{condiction1}
\\
&& A_0=T^{-1}A T  \label{condiction2},\\
&& C_0=C T  \label{condiction3},
 \end{eqnarray}
  find the  similarity matrix $T.$
 \end{problem}

 Rewrite \eqref{condiction2} in form
\begin{equation}\label{rlc25}
    TA_0=AT
\end{equation}
and after vectorization, we obtain:
\begin{eqnarray}
    \label{rlc26}
    \text{Vec}\left(TA_0\right)&=&\left(A_0^T \otimes I_{n_x} \right)\text{Vec}\left(T\right),\\
    \label{rlc27}
    \text{Vec}\left(AT\right)&=&\left(I_{n_x}\otimes A \right)\text{Vec}\left(T\right).
\end{eqnarray}
Equalizing  \eqref{rlc26} to \eqref{rlc27}:
\begin{eqnarray}\label{rlc28}
 &   \left(A_0^T \otimes I_{n_x} \right)\text{Vec}\left(T\right)= \left(I_{n_x}\otimes A \right)\text{Vec}\left(T\right) \Leftrightarrow\\
 \Leftrightarrow  &      \left(A_0^T \otimes I_{n_x}-I_{n_x}\otimes A \right)\text{Vec}\left(T\right) =0 \nonumber \\
\implies &\text{Vec}\left(T\right) \in\ker \left(A_0^T \otimes I_{n_x}-I_{n_x}\otimes A \right) .
\end{eqnarray}

\begin{lemma}
 $M_1,\;M_2 \in \mathbb{R}^{n\times n}$ are two similar  matrices, i.e., \begin{equation}\label{rlc29}
    M_2=T^{-1}M_1T,\qquad T\in\mathbb{R}^{n\times n}.
\end{equation}

Then
\begin{eqnarray}\label{rlc30}
  &  \text{rank}\left(I_{n_x} \otimes M_2-M_1^T\times I_{n_x} \right)={n_x}^2-{n_x} \Leftrightarrow \\
    \Leftrightarrow & \text{dim}\left\{\ker\left(I_{n_x} \otimes M_2-M_1^T\times I_{n_x} \right)\right\}={n_x}.
\end{eqnarray}
\end{lemma}
\begin{proof}
Consider $T_d \in\mathbb{R}^{{n_x} \times {n_x}}$ a singular matrix such that
\begin{equation}\label{rlc31}
    T_d^{-1}M_1^TT_d=\Lambda=
    \begin{bmatrix}
    \Lambda_1 & 0 & \cdots & 0 \\
     0 & \Lambda_2 & \cdots & 0 \\
     \vdots & \vdots & \ddots & 0 \\
     0 & 0 & \cdots & \Lambda_k
    \end{bmatrix},
\end{equation}
where $\Lambda_k$ are the blocks of a jordan canonical form of $M_1.$
Define $M:=\left(I_{n_x} \otimes M_2-M_1^T\otimes I_{n_x}\right) \in \mathbb{R}^{{n_x}^2\times {n_x}^2}.$ Multiply $M$  by $T_d^{-1} \otimes I_{n_x}$ on the left side and  by $T_d \otimes I_{n_x}$ on the right side. As
$    \text{rank}\left(T_d^{-1} \otimes I_{n_x}\right)=\text{rank}\left(T_d \otimes I_{n_x}\right)={n_x}^2 $
then
 $   \text{rank}\left(M\right)=\text{rank}\left\{\left(T_d^{-1} \otimes I_{n_x}\right) M \left(T_d \otimes I_{n_x}\right)\right\}.$
On the other side,
\begin{eqnarray}
&& \left(T_d^{-1} \otimes I_{n_x}\right) M \left(T_d \otimes I_{n_x}\right)\nonumber \\
  &=&  \left(T_d^{-1} \otimes I_{n_x}\right)\left(I_{n_x} \otimes M_2-M_1^T\otimes I_{n_x}\right)\left(T_d \otimes I_{n_x}\right) \nonumber \\
    &=&
    \left(T_d^{-1} \otimes I_{n_x}\right)\left(I_{n_x} \otimes M_2\right)\left(T_d \otimes I_{n_x}\right)-\nonumber \\
    & &
    - \left(T_d^{-1} \otimes I_{n_x}\right)\left(M_1^T\otimes I_{n_x}\right)\left(T_d \otimes I_{n_x}\right). \label{rlc34}
\end{eqnarray}
Define
$ T_d^{-1}:=
 \begin{bmatrix}
 \tau_{11} & \tau_{12} & \cdots & \tau_{1n} \\
 \tau_{21} & \tau_{22} & \cdots & \tau_{2n} \\
 \vdots    & \vdots    & \vdots & \vdots     \\
 \tau_{n1} & \tau_{n2} & \cdots & \tau_{nn}
 \end{bmatrix}$ and
using the

Kronecker product definition, hence
\begin{equation}\label{rlc36}
T_d^{-1} \otimes I_{n_x}=
\begin{bmatrix}
 \tau_{11}I_{n_x} & \tau_{12}I_{n_x} & \cdots & \tau_{1n}I_{n_x} \\
 \tau_{21}I_{n_x} & \tau_{22}I_{n_x} & \cdots & \tau_{2n}I_{n_x} \\
 \vdots    & \vdots    & \vdots & \vdots     \\
 \tau_{n1}I_{n_x} & \tau_{n2}I_{n_x} & \cdots & \tau_{nn}I_{n_x}
 \end{bmatrix}
\end{equation}
and
\begin{equation}\label{rlc37}
    I_{n_x} \otimes M_2=
    \begin{bmatrix}
        M_2 & 0 & \cdots & 0 \\
        0  & M_2 & \cdots & 0 \\
        \vdots & \vdots & \ddots & \vdots \\
        0 & 0 & \cdots & M_2
    \end{bmatrix}.
\end{equation}
Consequently
\begin{eqnarray*}\label{rlc38}
  &&  \left(T_d^{-1} \otimes I_{n_x}\right)\left(I_{n_x} \otimes M_2\right)=\\
 &=&   \begin{bmatrix}
       \tau_{11}I_{n_x} & \tau_{12}I_{n_x} & \cdots & \tau_{1n}I_{n_x} \\
       \tau_{21}I_{n_x} & \tau_{22}I_{n_x} & \cdots & \tau_{2n}I_{n_x} \\
       \vdots    & \vdots    & \vdots & \vdots     \\
       \tau_{n1}I_{n_x} & \tau_{n2}I_{n_x} & \cdots & \tau_{nn}I_{n_x}
    \end{bmatrix}
    \begin{bmatrix}
        M_2 & 0 & \cdots & 0 \\
        0  & M_2 & \cdots & 0 \\
        \vdots & \vdots & \ddots & \vdots \\
        0 & 0 & \cdots & M_2
    \end{bmatrix}  \\[10mm]
    &=&
    \begin{bmatrix}
       \tau_{11}M_2 & \tau_{12}M_2 & \cdots & \tau_{1n}M_2 \\
       \tau_{21}M_2 & \tau_{22}M_2 & \cdots & \tau_{2n}M_2 \\
       \vdots    & \vdots    & \vdots & \vdots     \\
       \tau_{n1}M_2 & \tau_{n2}M_2 & \cdots & \tau_{nn}M_2
    \end{bmatrix} =T_d^{-1} \otimes M_2
\end{eqnarray*}
and
\begin{eqnarray}
 && \left(T_d^{-1} \otimes I_{n_x}\right)\left(I_{n_x} \otimes M_2\right)\left(T_d \otimes I_{n_x}\right) \nonumber \\
  &=&
    \left(T_d^{-1} \otimes M_2\right)\left(T_d \otimes I_{n_x}\right.\nonumber\\
&=&  \left(T_d^{-1}T_d \right) \otimes M_2=I_{n_x} \otimes M_2.  \label{rlc39}
\end{eqnarray}
Using the Kronecker product properties and   $T_d$ definition:
\begin{eqnarray}\label{rlc42}
 &&    \left(T_d^{-1} \otimes I_{n_x}\right)\left(M_1^T\otimes I_{n_x}\right)\left(T_d \otimes I_{n_x}\right) \nonumber\\
    &=&
    \left[\left(T_d^{-1}M_1^T\right) \otimes I_{n_x}\right]\left(T_d^{-1} \otimes I_{n_x}\right) \nonumber \\
    &=&
    \left(T_d^{-1}M_1^TT_d^{-1}M_1^T\right) \otimes I_{n_x}=\Lambda \otimes I_{n_x}. \label{rlc41}
\end{eqnarray}
From \eqref{rlc34}, \eqref{rlc39} and  \eqref{rlc42} and the definitions of  $\Lambda$ and the  Kronecker product, we obtain:
\begin{eqnarray}\label{rlc43}
&& \left(T_d^{-1} \otimes I_{n_x}\right) M \left(T_d \otimes I_{n_x}\right)\nonumber\\
&=&I_{n_x} \otimes M_2-\Lambda \otimes I_{n_x} \nonumber\\
    &=&
  {\rm diag}  \left(  {\rm diag} \left(M_2 \right)_{j=1}^{n_i}-\Lambda_i\otimes I_{n_x}   \right)_{i=1}^k .
\end{eqnarray}
As $  {\rm diag} \left(M_2 \right)_{j=1}^{n_i}-\Lambda_i\otimes I_{n_x} $ is equal to
$$
\left[ \begin{array}{cccccc}
M_2 -\Lambda_i I_{n_x} & I_{n_x} & 0 & \cdots & \cdots& 0 \\
0& M_2 -\Lambda_i I_{n_x} & I_{n_x}& \cdots & \cdots & 0 \\
\vdots &  \ddots  &    & \ddots& & \vdots \\
\vdots &  &  \ddots   & & \ddots & \vdots \\
\vdots&  &  & & &I_{n_x} \\
0& 0  & 0 &\cdots &\cdots & M_2 -\Lambda_i I_{n_x}
 \end{array}\right]$$
 whose rank is $({n_x}-1)n_i.$  Hence
 \begin{equation}\label{Kalman}
 {\rm rank} \left(  {\rm diag}  \left(  {\rm diag} \left(M_2 \right)_{j=1}^{n_i}-\Lambda_i\otimes I_{n_x}   \right)_{i=1}^k  \right) = \sum_{i=1}^k ({n_x}-1)n_i
 \end{equation}
and  $\displaystyle  \sum_{i=1}^k ({n_x}-1)n_i= ({n_x}-1)  \sum_{i=1}^kn_i =  ({n_x}-1) {n_x} = n_x^2-n_x,$ because the sum of the dimension of the Jordan blocks of a matrix is the dimension of the matrix, i.e., $ \sum_{i=1}^k n_i = n_x.$
As a result ${\rm rank}\left(M\right)=n_x^2-n_x.$
\end{proof}
Back to our problem, we know that:
\begin{equation}\label{rlc45}
\text{Vec}\left(T\right) \in\ker \left(A_0^T \otimes I_{n_x}-I_{n_x}\otimes A \right)
\end{equation}
and
\begin{equation}\label{rlc46}
    \text{dim}\left\{\ker \left(A_0^T \otimes I_{n_x}-I_{n_x}\otimes A\right)\right)\}=n_x.
\end{equation}
Consider $U_i\in\mathbb{R}^{n_x \times n_x}$  matrices such that $$\left\{\text{vec}\left(U_1\right),\;\text{vec}\left(U_2\right),\dots,\;\text{vec}\left(U_{n_x}\right)\right\}$$ is a  basis of  $\ker \left(A_0^T \otimes I_{n_x}-I_{n_x}\otimes A\right)\}.$ Then
\begin{equation}\label{rlc47}
    T=\alpha_1 U_1+\alpha_2 U_2+\dots +\alpha_{n_x} U_{n_x},
\end{equation}
where $\alpha_i \in\mathbb{R},\;i=1,\dots,n_x.$ To calculate  $\alpha_i,\;i=1,\dots,n_x,$ we only have to solve equation  \eqref{rlc47}.  Substituing  $T$ according to \eqref{rlc47} into \eqref{condiction3}:
\begin{eqnarray}\label{rlc48}
    C_0&=&C\left(\alpha_1 U_1+\alpha_2 U_2+\dots +\alpha_{n_x} U_{n_x}\right) \nonumber \\
    &= &
    \begin{bmatrix}
       \alpha_1 & \alpha_2 & \cdots & \alpha_{n_x}
    \end{bmatrix}
    \begin{bmatrix}
       C U_1 \\ C U_2 \\ \vdots \\ C U_{n_x}
    \end{bmatrix}.
\end{eqnarray}
From the observability of $\left(A_0,C_0\right),$ $T$ is a unique nonsingular matrix.
As a result,
$\begin{bmatrix} C U_1 \\ C U_2 \\ \vdots \\ C U_{n_x}\end{bmatrix}$ is of full rank and
 $$   \begin{bmatrix}
       \alpha_1 & \alpha_2 & \cdots & \alpha_{n_x}
    \end{bmatrix}=
    C_0
    \left(
    \begin{bmatrix}
       C U_1 \\ C U_2 \\ \vdots \\ C U_{n_x}
    \end{bmatrix}
    \right)^{-1}_{.}$$

Once $T$ is known and since is unique, having the state space realization  $\left(A,B,C,D\right)$ such that the conditions of Problem~1 are met, it is also possible to calculate
    \begin{eqnarray}
       B_0&= &T^{-1}B, \\
    D_0&= &D .
     \end{eqnarray}
In this manner,  we obtain another minimal state-space realization, $\left(A_0,B_0,C_0,D_0\right),$ for the same LTI.

\end{document}